\documentclass{cedram-aif}
\usepackage{graphicx, epigraph}
\usepackage{amsmath,amsthm,amssymb,mathtools,amsfonts,mathrsfs}
\usepackage{xspace,xcolor}
\usepackage[breaklinks,colorlinks,citecolor=teal,linkcolor=teal,urlcolor=teal,pagebackref,hyperindex]{hyperref}
\usepackage[alphabetic]{amsrefs}
\usepackage{tikz-cd}
\usepackage[all]{xy}

\theoremstyle{definition}
\newtheorem{conv}{Convention}
\newtheorem{ques}{Question}
\equalenv{ex}{exam}
\equalenv{rmk}{rema}
\equalenv{cor}{coro}
\equalenv{lem}{lemm}

\newcommand{\C}{\mathbb C}
\newcommand{\Q}{\mathbb Q}
\newcommand{\Z}{\mathbb Z}

\newcommand{\Pp}{\mathbb P}
\newcommand{\G}{\mathbb G}

\newcommand{\sO}{\mathcal O}

\newcommand{\bM}{\overline{\mathcal M}}

\newcommand{\scrF}{\mathscr F}

\newcommand{\scrU}{\mathscr U}

\newcommand{\note}{note}

\title{A quantum splitting principle and an application}
\alttitle{Un principe de d\'ecomposition quantique et une application}

\author{\firstname{Honglu} \lastname{Fan}}

\address{ETH Z\"urich\\
  Department of Mathematics\\
  R\"amistrasse 101\\
Z\"urich, 8092 (Switzerland)}

\email{honglu.fan@math.ethz.ch}

\thanks{
The author would like to thank his advisor Yuan-Pin Lee for a few weeks'
discussions about this work. In addition, the application part would not have
been completed without discussions with Junliang Shen. He would also like to
thank Feng Qu for helpful discussions, and Pierrick Bousseau for French translation. 
A large part of the work was finished
during a pleasant visit in National Taiwan University. The author had been
supported by NSF and grant ERC-2012-AdG-320368-MCSK in the group of Rahul
Pandharipande and is currently supported by SwissMAP.
}

\altkeywords{th\'eorie de Gromov--Witten, principe de d\'ecomposition, fibr\'e projectif}


\begin{document}

\begin{abstract}
We propose an analogy of splitting principle in genus-$0$ Gromov--Witten theory. More precisely, we show how the Gromov--Witten theory of a variety $X$ can be embedded into the theory of the projectivization of a vector bundle over $X$. An application is also given.
\end{abstract}
\begin{altabstract}
Nous proposons un analogue du principe de d\'ecomposition en th\'eorie de Gromov--Witten de genre z\'ero. Plus pr\'ecis\'ement, nous montrons comment r\'ealiser la th\'eorie de Gromov--Witten d'une vari\'et\'e X dans la th\'eorie de la projectivisation d'un fibr\'e vectoriel sur X. Nous donnons \'egalement une application.
\end{altabstract}

\maketitle

\setcounter{section}{-1}

\section{Introduction}
In geometry and topology, it is common to apply base-change in order to reduce a problem. For example, splitting principle in algebraic topology reduces a vector bundle to a split one by applying a base-change. More specifically, let $X$ be a smooth manifold and $V$ a vector bundle over $X$. The pull-back of $V$ along the projection $\pi: \Pp_X(V)\rightarrow X$ splits off a line bundle. Applying it multiple times eventually reduces $V$ to a split bundle. But to make it work, facts like the following are crucial.
\[
\text{The pull-back~}\pi^*: H^*(X;\Q) \rightarrow H^*(\Pp_X(V);\Q) \text{~is injective}.
\]
By using this, a lot of cohomological computations that is done on the pull-back of $V$ can be related back to the ones on $V$.

In this paper, our focus will be on Gromov--Witten theory. But the above technique cannot be applied easily because Gromov--Witten theory lacks a nice "pull-back" operation. Let $f:Y\rightarrow X$ be a map between smooth projective varieties. The complexity already appears when $f$ is an embedding. In this case, one might na\"ively hope for some statement in the spirit of the following.
\[
\text{(a part of) GWT}\text{($Y$) is a sub-theory of GWT}\text{($X$)},
\]
where GWT stands for Gromov--Witten theory. But this is not quite true. The
closest theorem might be the quantum Lefschetz theorem (\cites{CG,KKP}), which
only works in genus $0$ and has a restriction on the embedding.

Another common situation is when $f:Y\rightarrow X$ is a smooth morphism. When the fibration is nice (e.g. projectivization of vector bundles in splitting principle), one might wonder whether something like the following could happen.
\[
\text{GWT}\text{($X$) is a sub-theory of GWT}\text{($Y$)}.
\]
If a theorem with the above idea can be established, we might be able to imitate splitting principle to simplify some problems via base-changes. But again, the relation between GWT($X$) and GWT($Y$) seems complicated, and very few work has been done along these lines. 

We will work on the case when $Y$ is the projectivization of a vector bundle $V$ over $X$. In this paper, Corollary \ref{cor:main} realizes genus-$0$ GWT($X$) as a sub-theory of genus-$0$ GWT($\Pp(V)$) in a reasonably simple statement. To show an application, we generalize the main theorem in \cite{F} to $\Pp^n$-fibrations in genus $0$.

It's worth noting that splitting principle has already inspired some of the techniques in \cite{LLQW}. Roughly speaking, they apply base-changes along blow-ups to turn a vector bundle into a split one, and relate back via degeneration formula. But the relate-back process is very complicated, mainly due to the complexity of degeneration formula.

\subsection{Statements of results}
Let $X$ be a smooth projective variety, $V$ be a vector bundle over $X$ and
$\beta \in NE(X)$ where $NE(X)$ stands for the group of numerical curve classes.
We first prove the following theorem.
\begin{theo}\label{thm:main0}
If $V$ is globally generated, then the following holds.
\begin{equation*}
\langle \pi^*\sigma_1,\ldots, \pi^*\sigma_n \rangle_{0,n,\tilde\beta}^{\Pp_X(V), \sO(-1)}=\langle \sigma_1,\ldots, \sigma_n \rangle_{0,n,\beta}^{X, V},
\end{equation*}
where $\tilde\beta\in NE(\Pp(V))$ is an effective class such that $\pi_*\tilde\beta=\beta$ and $(\tilde\beta,c_1(\sO(1)))=0$.
\end{theo}
Here both sides of the equation are twisted Gromov--Witten invariants defined in
Section \ref{sec:twistedgw} equation \eqref{eqn:twistedgw}.
\begin{rmk}
If we put descendants into the invariants and match them up like in Theorem \ref{thm:main0}, the equality might \emph{not} hold. But it holds if we use the pull-back of $\psi_i$ in $\bM_{0,n}(X,\beta)$ on the left-hand side. The difference between $\psi$-classes and the pull-back of $\psi$-classes can be made explicit, but it's irrelevant in our paper.
\end{rmk}

Let $\lambda$ be the equivariant parameter in the twisted theory. We also use the notation $[\dotsb]_{\lambda^{-N}}$ for taking the $\lambda^{-N}$ coefficient. We have the following corollary.
\begin{cor}\label{cor:main}
If $V$ is globally generated, Gromov--Witten invariants of $X$ can be computed from that of $Y$ according to the following.
\[
\langle \sigma_1,\ldots, \sigma_n \rangle_{0,n,\beta}^{X}=\left[ \langle \pi^*\sigma_1,\ldots, \pi^*\sigma_n \rangle_{0,n,\tilde\beta}^{\Pp_X(V), \sO(-1)} \right]_{\lambda^{-N}},
\]
where $N=rank(V)+(\tilde\beta,det(T_\pi))$ with $T_\pi$ the relative tangent bundle, 
\end{cor}

%We will generalize this corollary further into the following.

%\begin{theo}\label{thm:main}
%Let $\pi: Y\rightarrow X$ be a smooth morphism between smooth projective varieties such that each fiber is isomorphic to $\Pp^n$. Let $D\in Pic(Y)$ be a divisor that restricts to $\sO(1)$ on each fiber. Then there is a sufficiently ample line bundle $L\in Pic(X)$ such that the following holds.
%\[
%	\langle \sigma_1,\dotsc,\sigma_n \rangle_{0,n,\beta} = \left[ \langle \pi^*\sigma_1,\ldots, \pi^*\sigma_n \rangle_{0,n,\tilde\beta}^{Y, \sO(-D+\pi^*L)} \right]_{\lambda^N},
%\]
%where $N=dim(Y)-dim(X)+(\tilde\beta,det(T_\pi))$ with $T_\pi$ the relative tangent bundle, and $\tilde\beta\in NE(X)$ is an effective class such that $\pi_*\tilde\beta=\beta$ and $(\tilde\beta,D-\pi^*L)=0$.
%\end{theo}

%\emph{****Such a D does not exist!!!****}

As an application, we prove the following generalization of \cite{F} in genus $0$.
\begin{theo}\label{thm:main1}
Let $\pi_1:P_1\rightarrow X$, $\pi_2:P_2\rightarrow X$ be two projective smooth morphisms between smooth projective varieties whose fibers are isomorphic to $\Pp^n$. Suppose their Brauer classes are equal in $H^2_{\text{\'et}}(X,\mathbb G_m)$. Furthermore, suppose there exists a ring isomorphism 
\[
\scrF:H^*(P_1;\Q) \rightarrow H^*(P_2;\Q),
\]
such that
\begin{enumerate}
\item $\scrF$ sends the subgroup $H^*(P_1;\Z)/tor(H^*(P_1;\Z))$ into $H^*(P_2;\Z)/tor(H^*(P_2;\Z))$. Here $tor(-)$ denotes the torsion subgroup.
\item $\scrF$ restricts to identity on $H^*(X;\Q)$. Here we identify $H^*(X;\Q)$ as subrings under pull-backs of $\pi_1,\pi_2$.
\item $\scrF(c_1(\omega_{\pi_1}))=c_1(\omega_{\pi_2})$, where $\omega_{\pi_i}$ is the corresponding relative canonical sheaf.
\end{enumerate}
Then we have
\[
\langle \psi^{k_1}\sigma_1,\dotsc,\psi^{k_n}\sigma_n \rangle_{0,n,\beta}^{P_1}=\langle \psi^{k_1}\scrF\sigma_1,\dotsc,\psi^{k_n}\scrF\sigma_n \rangle_{0,n, \Psi(\beta)}^{P_2},
\]
where $\Psi:N_1(P_1)\rightarrow N_1(P_2)$ is the isomorphism induced by $\scrF$ under Poincar\'e duality.  
\end{theo}

\begin{rmk}
When one of $P_i$ is the projectivization of vector bundle, all these conditions are equivalent to the existence of vector bundles $V_1, V_2$ such that $c(V_1)=c(V_2)$ and $P_1=\Pp(V_1), P_2=\Pp(V_2)$. This will be explained in Lemma \ref{lem:projbdl}.
\end{rmk}
\begin{rmk}
The reduction technique in \cite{LLQW} by blow-ups does not seem to work in our situation, as the Brauer group is a birational invariant.
\end{rmk}

\subsection{Some further questions}
Our result is based on Theorem \ref{thm:vircls}. Note that this method seems to fail in high genus. Therefore a further question is the following.
\begin{ques}
Are there any similar results as Theorem \ref{thm:main0} or Corollary \ref{cor:main} in high genus?
\end{ques}

Although the proof of Theorem \ref{thm:main1} is to mainly demonstrate the use of ``splitting principle" type of idea in Gromov--Witten theory, we are still interested in the following question.
\begin{ques}
Is it essential to put the requirement that $P_1, P_2$ have the same Brauer class?
\end{ques}
%Note that $\scrF$ already forces the image of their Brauer classes in $H^3(X,\Z)$ to be the same. We don't know whether it further implies that $\alpha(P_1)=\alpha(P_2)\in H^2_{\text{\'et}}(X,\mathbb G_m)$. 
In the meantime, we are wondering to what extent the topology can determine the GWT of a smooth fibration. In particular, one way to phrase our questions is the following.
\begin{ques}
Let $P_1\rightarrow X, P_2\rightarrow X$ be two smooth projective morphisms between smooth projective varieties.
Suppose all of their fibers are isomorphic to a certain variety $F$. Can we further impose a suitable isomorphism between $H^*(P_1,\Q)$ and $H^*(P_2,\Q)$, and conclude GWT($P_1$)=GWT($P_2$)?
\end{ques}

Finally, we want to remark that Theorem \ref{thm:vircls2} yields a number of new identities besides Thereom \ref{thm:main0}. They are summarized in section \ref{section:thmmain0}. They are interesting to us for a few reasons. For example they are relatively simple, especially because they are closed formulas and there is no generating series or mirror map involved (compare \cite{GB}). This seems to show us that the Gromov--Witten theory of $\Pp(V)$ (maybe twisted by $\sO(-1)$) and the Gromov--Witten theory of $X$ twisted by $V$ are closely related. Note that another evidence that already exists in literature is the striking similarity between their $I$-functions when $V$ splits. We suspect that these identities are a tip of the iceberg, and there might be a general relationship underneath. Therefore we would like to ask the following (vague) question.
\begin{ques}
Is there a reasonably nice relationship between GWT($\Pp(V)$) and the GWT(X) twisted by $V$, that includes our identities as special cases?
\end{ques}

\subsection{Structure of the paper}

In order to prove Theorem \ref{thm:main0}, we need to take a detour to the naturality problem of Gromov--Witten invariants under blow-up proposed by Y. Ruan in \cite{R}. A variation of the main theorem in \cite{L} under the equivariant setting (Theorem \ref{thm:vircls2}) is stated in section \ref{sec:constr}. Its weaker version (Theorem \ref{thm:vircls}) is finally proven in section \ref{sec:virpull}, and this already implies Theorem \ref{thm:main0} by comparing localization residues (section \ref{section:thmmain0}). For completeness, we finish the proof of Theorem \ref{thm:vircls2} in section \ref{sec:vircls2}. Finally as an application, we prove Theorem \ref{thm:main1} in section \ref{sec:Pnfibr}.

\section{The naturality of virtual class under certain blow-up}\label{sec:constr}
In this section, we take a detour to study the naturality of virtual class under blow-up. Keep in mind that our ultimate goal is still Theorem \ref{thm:main0}, and it will be proven in section \ref{section:thmmain0} when the tools are ready.

Recall a vector bundle $V$ over a variety $X$ is \emph{convex} if for any morphism $f:\Pp^1\rightarrow X$, $H^1(\Pp^1,f^*V)=0$. A smooth projective variety is called \emph{convex} if its tangent bundle is a convex vector bundle. Note that a globally generated vector bundle is always convex.

Let $X$ be a smooth projective variety. Let $V$ be a vector bundle of rank $\geq 2$ over $X$. We want to understand the Gromov--Witten theory of $\Pp(V)$. For our purpose, we will assume the following.
\begin{conv}
Assume $V$ is a globally generated vector bundle.
\end{conv}
This assumption can be realized by tensoring with a sufficiently ample line bundle. Define $Y=\Pp(V\oplus \sO)$. This projective bundle can be viewed as a compactification of the total space of $V$. The zero section induces an embedding $X\hookrightarrow Y$. We denote its image by $Y_0$. There is also an embedding $\Pp(V)\hookrightarrow Y$ whose image is denoted by $Y_\infty$ (the divisor at infinity).

$\C^*$ can act on $V$ by fixing the base and scaling the fibers (direct sum of weight $1$ representations). This will induce a $\C^*$ action on $Y$. Consider $\tilde Y=Bl_{Y_0}Y$. We denote the contraction by
\[
	\pi:\tilde Y\rightarrow Y.
\]
The $\C^*$ action uniquely lifts to $\tilde Y$. On the other hand, it is not hard to see \[\tilde Y\cong \Pp_{\Pp(V)}(\sO(-1)\oplus \sO).\] Again, it can be seen as the compactification of the line bundle $\sO(-1)$ over $\Pp(V)$. Denote the zero section by $\tilde Y_0$ and the infinity section by $\tilde Y_\infty$. One can check that \[\pi(\tilde Y_0)=Y_0 \quad,\quad \pi(\tilde Y_\infty)=Y_\infty.\]

This $\C^*$ action induces actions on moduli of stable maps. There is a stabilization morphism $\tau:\bM_{0,n}(\tilde Y,\pi^!\beta) \rightarrow \bM_{0,n}(Y,\beta)$. Because $V$ is globally generated, by \cite{L} we have (nonequivariantly)
\[
\tau_*[\bM_{0,n}(\tilde Y,\pi^!\beta)]^{vir}=[\bM_{0,n}(Y,\beta)]^{vir}.
\]
In this \note, we would like to partially extend the above identity to equivariant setting under the given $\C^*$ action. %Therefore we will make the following assumption.
%\begin{conv}
%In this \note, a virtual class will be assumed to be equivariant under the $\C^*$ action given above unless otherwise stated.
%\end{conv}
Let $Z\subset \bM_{0,n}(Y,\beta)$ be the locus consisting of the stable maps at least one of whose components map nontrivially to the closed subvariety $Y_\infty$. $Z$ is a closed substack. Denote $\scrU\subset \bM_{0,n}(Y,\beta)\backslash Z$ the complement. Let $[\scrU]^{vir,eq}$ denote the equivariant virtual class induced by the perfect obstruction theory coming from the moduli of stable map. Also denote $\tau^{-1}(\scrU)$ by $\tilde\scrU$. We will show the following.
\begin{theo}\label{thm:vircls}
The following holds for equivariant virtual classes.
\[
\tau_*[\tilde\scrU]^{vir,eq}=[\scrU]^{vir,eq}.
\]
\end{theo}

%This theorem will imply a number of interesting identities of genus-$0$ Gromov--Witten invariants of $\Pp(V)$ (See Section \ref{sec:consq}). These identities seem to be new in literature, and they make up the highlights in this \note.

Once this is established, we can compare the localization residues, and eventually prove the following by virtual localization.

\begin{theo}\label{thm:vircls2}
The following holds for equivariant virtual classes.
\[
\tau_*[\bM_{0,n}(\tilde Y,\pi^!\beta)]^{vir,eq}=[\bM_{0,n}(Y,\beta)]^{vir,eq}.
\]
\end{theo}

We prove Theorem \ref{thm:vircls2} for the sake of completeness. But Theorem \ref{thm:vircls} is already strong enough to deduce Theorem \ref{thm:main0}.
%Theorem \ref{thm:vircls2} may not sound too surprising, as we require $V$ to be globally generated which is stronger than the condition in \cite{L}. But in our equivariant case, the perturbation technique in \cite{L} is no longer available. But it is an interesting question whether one can prove it in a more conceptual way (like the technique developed in \cite{M}).

\subsection{Towards the proof of Theorem \ref{thm:vircls}}
We start with the following lemma.
\begin{lem}\label{lem:embed}
Under the setting of Section \ref{sec:constr}, there exists a convex projective variety $\G$, an embedding $i:X\rightarrow\G$ and a globally generated vector bundle $V_\G$ such that $i^*V_\G=V$.
\end{lem}
\begin{proof}
First of all because $V$ is globally generated, there is a map $f:X\rightarrow G(k,n)$ to some Grassmannian such that $f^*Q=V$ where $Q$ is the universal quotient bundle over $G(k,n)$. Notice $Q$ is globally generated. The only thing different in this situation is that $f$ may not be an embedding. However, we can pick an embedding $f':X\rightarrow \Pp^N$ for some N. Now they induce an embedding $f\times f':X\rightarrow G(k,n)\times \Pp^N$. It is not hard to see that $(f\times f')^* (Q\boxtimes \sO_{\Pp^N})=V$. $Q\boxtimes \sO_{\Pp^N}$ is still globally generated since it is a quotient of direct sums of trivial bundles. Let $\G=G(k,n)\times \Pp^N$ and $V_\G=Q\boxtimes \sO_{\Pp^N}$.
\end{proof}

Denote $Y_\G=\Pp_\G(V_\G\oplus \sO)$ and $\tilde Y_\G=\Pp_{\Pp_\G(V_\G\oplus\sO)}(\sO(-1)\oplus \sO)$. Similar as before, $\tilde Y_\G$ is the blow-up of $Y_\G$ at zero section. Let $Z_{\G}\subset \bM_{0,n}(Y_\G,\beta)$ be the closed substack containing stable maps with at least one component mapping nontrivially onto $(Y_{\G})_\infty$, the image of the embedding $\Pp_\G(V_\G)\subset \Pp_\G(V_\G\oplus\sO)$. Similarly, we denote $\tau_\G:\bM_{0,n}(\tilde Y_\G,\pi^!\beta)\rightarrow \bM_{0,n}(Y_\G,\beta)$, $\scrU_\G=\bM_{0,n}(Y_\G,\beta)\backslash Z_\G$ and $\tilde\scrU_\G=\tau_\G^{-1}(\scrU_\G)$. We have a similar $\C^*$ action on $Y_\G$, $\tilde Y_\G$ and their associated moduli of stable maps like before.
\begin{lem}\label{lem:convex}
Theorem \ref{thm:vircls} is true when $X=\G$ and $V=V_{\G}$.
\end{lem}
Its proof is given in Section \ref{sec:convex}. In Section \ref{sec:virpull} we show the above lemma implies Theorem \ref{thm:vircls}.

\subsection{The case when $X$ is convex}\label{sec:convex}
The argument here is basically parallel to the proof of \cite[Theorem 1.4]{L}. The situation here is even simpler, since the ``Assumption *" in \cite{L} is automatically satisfied because of the following. 
\begin{lem}\label{lem:unobs}
The open substack $\scrU_\G\subset\bM_{0,n}(Y_\G,\beta)$ is unobstructed, thus smooth.
\end{lem}
\begin{proof}
It suffices to prove that for any stable map $(C,x_1,\ldots,x_n,f)\in \scrU_\G$, $H^1(C,f^*T{Y_\G})=0$. By passing to normalization, one may assume $C=\Pp^1$. We have the following Euler sequence of $Y_\G$.
\[
0\rightarrow \sO \rightarrow (p_\G)^*(V\oplus \sO)\oplus \sO_{\Pp_\G(V\oplus\sO)}(1) \rightarrow T_{p_\G} \rightarrow 0,
\]
where $T_{p_\G}$ is the relative tangent bundle of $p_\G:Y_\G\rightarrow \G$ the projection. We also have the short exact sequence
\[
0\rightarrow T_{p_\G}\rightarrow T_{Y_\G} \rightarrow (p_\G)^*T_\G \rightarrow 0.
\]
Since they are vector bundles, this short exact sequence pulls back to
\[
0\rightarrow f^*T_{p_\G}\rightarrow f^*T_{Y_\G} \rightarrow f^*(p_\G)^*T_\G \rightarrow 0
\]
over $\Pp^1$. Now $f^*(p_\G)^*T_\G$ is already convex. If we can prove $f^*T_{p_\G}$ is also convex, then $f^*T_{Y_\G}$ will be convex as well. We will show that there is a point $x\in \Pp^1$ such that the global sections of $T_{p_\G}$ generates its fiber at $f(x)$. One point is enough because if a vector bundle over $\Pp^1$ has a negative factor, the global sections would not generate any fiber over $\Pp^1$.

We choose the point $x$ to be one that $f(x)\not\in (Y_\G)_\infty$ (by definition of $\scrU_\G$, such a $x$ exists). $f(x)$ could a priori be any point in $Y_\G\backslash (Y_\G)_\infty$. In the meantime, $T_{p_\G}$ is a quotient of $(p_\G)^*(V\oplus \sO)\oplus \sO_{\Pp_\G(V\oplus\sO)}(1)$. Therefore it suffices to show that global sections of $(p_\G)^*(V\oplus \sO)\oplus \sO_{\Pp_\G(V\oplus\sO)}(1)$ generates the fiber over any point in $Y_\G\backslash (Y_\G)_\infty$. This can be easily seen because 
\begin{enumerate}
\item $(p_\G)^*(V\oplus\sO)$ is globally generated;
\item $\sO_{\Pp_\G(V\oplus\sO)}(1)$ has a section $s$ whose vanishing locus is exactly $(Y_\G)_\infty$.
\end{enumerate}
One can multiply suitable sections with the section $s$ to generate the fiber of any given point in $Y_\G\backslash (Y_\G)_\infty$.
\end{proof}
Therefore we have $[\scrU]^{vir}=[\scrU]$. This is also why we choose this open substack. %Some details will be shown in the rest of the subsection.

For the rest of this subsection, we assume $X=\G$ is a convex variety. Let $U=Y\backslash Y_0$. Since $U$ doesn't intersect with the blow-up center, it is also an open subset of $\tilde Y$.
\begin{lem}
The open substack $\scrU\bigcap\bM_{0,n}(U,\beta) \subset \scrU$ is dense.
\end{lem}
\begin{proof}
Recall $V$ is globally generated. Since the rank of $V$ is larger than $1$, a general section is disjoint from the zero section. Choose any section $s_V:X\rightarrow V$ that is disjoint from the zero section. Also note that $N_{Y_0/Y}=V$. We claim that this section $s_V$ of $N_{Y_0/Y}$ lifts to the tangent space $TY$. It's easy to check that $N_{Y_0/Y}=T_p|_{Y_0}$ the restriction of relative tangent. By the Euler sequence of the relative tangent, we can check that \[(s_V\oplus 1)\otimes s\in \Gamma(p^*(V\oplus \sO)\otimes \sO_{\Pp_X(V\oplus \sO)}(1))\] is a lifting of $s_V$, where $s$ is the section of $\sO_{\Pp_X(V\oplus \sO)}(1)$ that vanishes along $Y_\infty$.

Now that $\scrU$ is unobstructed, for any point $(C,x_1,\ldots,x_n,f)\in \scrU$, the pull-back of this tangent vector field along $f$ is integrable and there is a one-parameter family whose generic member is in $\scrU\bigcap\bM_{0,n}(U,\beta)$.
%In order to prove the denseness, we show that for any stable map $(C,x_1,\ldots,x_n,f)\not\in \bM_{0,n}(U,\beta)$, there exists a family of stable maps $(\mathscr C, \tilde x_1,\ldots,\tilde x_n, \tilde f)$ over $\A^1$ such that the fiber at $0$ is $(f:C\rightarrow X)$ while any other fiber is a stable map in $\bM_{0,n}(U,\beta)$.
%Recall $V$ is globally generated. Since the rank of $V$ is larger than $1$, a general section is disjoint from the zero section. Choose any section $s:X\rightarrow V$ that is disjoint from the zero section. Notice that the composition $s\circ f:C\rightarrow V$ is a stable map whose image lies in $U$. Thus $(C,x_1,\ldots,x_n,s\circ f)$ lies in $\bM_{0,n}(U,\beta)$. A non-zero multiple of the section $s$ is a section that is still disjoint from zero section. Therefore such a family of stable maps can be constructed by composing with scalings of $s$.
\end{proof}

Let $d=dim\scrU=dim(\bM_{0,n}(Y,\beta))$ and $\mathcal Z=\scrU \backslash \scrU\bigcap \bM_{0,n}(U,\beta)$. The previous lemma implies \[dim(\mathcal Z) < d.\] We have the exact sequence
\[
A^{\C^*}_{d}(\mathcal Z) \rightarrow A^{\C^*}_{d}(\scrU) \rightarrow A^{\C^*}_{d}(\scrU\bigcap \bM_{0,n}(U,\beta)) \rightarrow 0.
\]
Recall that $A^{\C^*}_*(-)$ is bounded above by the dimension, even if it is an equivariant Chow group. Therefore the first term $A^{\C^*}_{d}(\mathcal Z)=0$ for dimensional reason. 

On the other hand, because of the open immersion $U\subset \tilde Y$, $\bM_{0,n}(U,\beta)$ has an open immersion into $\bM_{0,n}(\tilde Y,\pi^!\beta)$ (may not be dense). We have a similar exact sequence for $A^{\C^*}_{d}(\tilde\scrU)$, except that the first term may not vanish (we don't know the dimension of $\tau^{-1}\mathcal Z$). Because the blow-up $\pi$ restricts to isomorphism on the open subset $U\subset Y$, we have 
\[
\tilde\scrU\bigcap \bM_{0,n}(U,\beta) \cong \scrU\bigcap \bM_{0,n}(U,\beta),
\]
where we slightly abuse notations by regarding $U$ as a subset of $\tilde Y$ on the
left hand side and regarding $U$ as a subset of $Y$ on the right hand side. To put the two exact sequences together, we have
%\begin{equation}
\[
\xymatrix{
&A^{\C^*}_{d}(\tilde\scrU) \ar[r] \ar[d]^{\tau_*} & A^{\C^*}_{d}(\tilde\scrU\bigcap \bM_{0,n}(U,\beta)) \ar[r] \ar[d]^{\cong} & 0 \\
0 \ar[r] & A^{\C^*}_{d}(\scrU) \ar[r]^-{\cong} & A^{\C^*}_{d}(\scrU\bigcap \bM_{0,n}(U,\beta)) \ar[r] & 0 
}
\]
%\end{equation}
%SOME FURTHER ARGUMENTS
%Let $i_U:U\rightarrow \tilde Y$ be the open immersion, and $\bar i_U:\bM_{0,n}(U,\beta) \rightarrow \bM_{0,n}(\tilde Y,\pi^!\beta)$ be the
%induced open immersion between moduli of stable maps. Although $\bar i_U$ may
%not be a dense immersion, the virtual class is pulled back to virtual class. One
%can see it by observing Gysin pull-back commutes with open immersions.
Recall we want to show that $\tau_*[\tilde\scrU]^{vir,eq}=[\scrU]^{vir,eq}$ (via the left
column). If we let the class $[\tilde\scrU]^{vir,eq}$ go through the arrow
to the right, we notice that $[\tilde\scrU]^{vir,eq}$ is sent to
$[\tilde\scrU\bigcap \bM_{0,n}(U,\beta)]^{vir,eq}$ because virtual class is
pulled back to virtual class via open immersions. Then we proceed by going through the right
vertical arrow and the inverse of the bottom horizontal arrow. Since they are
isomorphisms, we easily conclude that Lemma \ref{lem:convex} is true.

\subsection{Virtual pull-back}\label{sec:virpull}
We are under the setting of Lemma \ref{lem:embed}.  There is the following diagram.
\begin{equation}\label{eqn:diag}
\xymatrix{
\bM_{0,n}(\tilde Y,\pi^!\beta) \ar[r]^{\tilde\iota} \ar[d]^{\tau} & \bM_{0,n}(\tilde Y_\G,\pi^!\beta) \ar[d]^{\tau_\G} \\
\bM_{0,n}(Y,\beta) \ar[r]^{\iota} & \bM_{0,n}(Y_\G,\beta). \\
}
\end{equation}
%\begin{lem}
%$Y_\G$ is convex.
%\end{lem}
%\begin{proof}
%\end{proof}

It is Cartesian by \cite[Remark 5.7]{M}. By base-change to an open substack, one has the following Cartesian diagram
\begin{equation}\label{eqn:diagU}
\xymatrix{
\tilde\scrU \ar[r]^{\tilde\iota} \ar[d]^{\tau} & \tilde\scrU_\G \ar[d]^{\tau_\G} \\
\scrU \ar[r]^{\iota} & \scrU_\G, \\
}
\end{equation}
where we slightly abuse notations by using the same $\iota, \tilde\iota, \tau, \tau_\G$ on open substacks.

In \cite{M}, virtual pull-backs are constructed (in our case, \cite[Construction 3.13]{M} applies). Note that for a torus-equivariant DM-type morphism, if there is an equivariant perfect relative obstruction theory, then the construction in \cite{M} can be done in the equivariant Chow in parallel. In our case, we denote the following virtual pull-backs under equivariant context.
\[
%\iota^!:A_*^{\C^*}(\bM_{0,n}(Y_\G,\beta)) \rightarrow A_{*-\delta}^{\C^*}(\bM_{0,n}(Y,\beta)),
\iota^!:A_*^{\C^*}(\scrU_\G) \rightarrow A_{*-\delta}^{\C^*}(\scrU),
\]
\[
%\tilde \iota^!:A_*^{\C^*}(\bM_{0,n}(\tilde Y_\G,\pi^!\beta)) \rightarrow A_{*-\delta'}^{\C^*}(\bM_{0,n}(\tilde Y,\pi^!\beta)),
\tilde \iota^!:A_*^{\C^*}(\tilde\scrU_\G) \rightarrow A_{*-\delta'}^{\C^*}(\tilde\scrU),
\]
where $\delta,\delta'$ are the differences of the corresponding virtual dimensions. Note that they are defined because the relative obstruction theory of $\iota$ is perfect over $\scrU$, which further follows from Lemma \ref{lem:unobs} under a similar argument to \cite[Remark 3.15]{M}. By \cite[Corollary 4.9]{M}, these morphisms send virtual classes to virtual classes.

\begin{lem}
Lemma \ref{lem:convex} implies Theorem \ref{thm:vircls}.
\end{lem}

It follows from a completely parallel argument to \cite[Proposition 5.14]{M}. Briefly speaking, one first check that the diagram \eqref{eqn:diag} is Cartesian. And then, 
%\begin{align}
%\begin{split}
%& \tau_*[\bM_{0,n}(\tilde Y,\pi^!\beta)]^{vir} \\
%=& \tau_*\tilde\iota^! [\bM_{0,n}(\tilde Y_\G,\pi^!\beta)]^{vir} \\
%=& \iota^! (\tau_\G)_* [\bM_{0,n}(\tilde Y_\G,\pi^!\beta)]^{vir} \\
%=& \iota^! [\bM_{0,n}(Y_\G,\beta)] \\
%=& [\bM_{0,n}(Y,\pi^!\beta)]^{vir}.
%\end{split}
%\end{align}
\begin{equation}
%\begin{split}
\tau_*[\tilde\scrU]^{vir,eq} = \tau_*\tilde\iota^! [\tilde\scrU_\G]^{vir,eq} = \iota^! (\tau_\G)_* [\tilde\scrU_\G]^{vir,eq} = \iota^! [\scrU_\G] = [\scrU]^{vir,eq}.
%\end{split}
\end{equation}
For the reason of each equality, one is also referred to the proof of \cite[Proposition 5.14]{M}.

\section{Consequences via localization}\label{sec:consq}

We would like to briefly recall the virtual localization in \cite{GP} and make the set-up in the next subsection.

Let $\mathscr X$ be a smooth projective variety admitting an action by a torus
$T=(\C^*)^m$. It induces an action of $T$ on $\bM_{g,n}(\mathscr X,\beta)$. Let
$\bM_\alpha$ be the connected components of the fixed locus $\bM_{g,n}(\mathscr
X,\beta)^T$ labeled by $\alpha$ with the inclusion $i_\alpha:
\bM_\alpha\rightarrow \bM_{g,n}(\mathscr X,\beta)$. The virtual fundamental
class $[\bM_{g,n}(\mathscr X,\beta)]^{vir}$ can be written as
\[
[\bM_{g,n}(\mathscr X,\beta)]^{vir}=\sum_\alpha (i_\alpha)_*  \displaystyle\frac{[\bM_\alpha]^{vir}}{e_T(N^{vir}_\alpha)},
\]
where $[\bM_\alpha]^{vir}$ is constructed from the fixed part of the restriction
of the perfect obstruction theory of $\bM_{g,n}(\mathscr X,\beta)$, the virtual
normal bundle $N^{vir}_\alpha$ is the moving part of the two term complex in the
perfect obstruction theory of $\bM_{g,n}(\mathscr X,\beta)$, and $e_T$ stands
for the equivariant Euler class. Sometimes we call
$\displaystyle\frac{[\bM_\alpha]^{vir}}{e_T(N^{vir}_\alpha)}$ the localization
residue of $\alpha$. In Gromov--Witten theory, one ingredient of these localization residues is the twisted theory.

\subsection{Twisted Gromov--Witten theory}\label{sec:twistedgw}
Let $X$ be a smooth projective variety, $E$ be a vector bundle over $X$. Let 
\[
ft_{n+1}:\bM_{g,n+1}(X,\beta) \rightarrow \bM_{g,n}(X,\beta)
\]
be the map forgetting the last marked point. Under this map $\bM_{g,n+1}(X,\beta)$ can be viewed as the universal family over $\bM_{g,n}(X,\beta)$. Furthermore, the evaluation map of the last marked point
\[
ev_{n+1}:\bM_{g,n+1}(X,\beta) \rightarrow X
\]
serves as the universal stable map from the universal family over $\bM_{g,n}(X,\beta)$.

Let $\C^*$ act on $X$ trivially and on $E$ by scaling (weight $1$ on every $1$-dimensional linear subspace of a fiber). Denote $\lambda$ the corresponding equivariant parameter. Let $E_{g,n,\beta}=[R(ft_{n+1})_*ev_{n+1}^*E]\in K^0_{\C^*}(\bM_{g,n}(X,\beta))$. This class can be represented by the difference of two vector bundles $E_{g,n,\beta}^0-E_{g,n,\beta}^1$ such that $\C^*$ acts on each of them by scaling as well.
%There exists a two-term complex of vector bundles in $\bM_{g,n}(X,\beta)$
%\[
%0\rightarrow E_{g,n,\beta}^0 \rightarrow E_{g,n,\beta}^1\rightarrow 0
%\]
%such that the $i$-th cohomology is $R^i(ft_{n+1})_*ev_{n+1}^*E$ for $(i=0,1)$.

%Let $\C^*$ act on $E_{g,n,\beta}^i$ by scaling as well. Write $E_{g,n,\beta}$ for the two-term complex $[E_{g,n,\beta}^0 \rightarrow E_{g,n,\beta}^1]$ in $D^b(\bM_{g,n}(X,\beta))$. Define the equivariant Euler class 
%\[
%e_{\C^*}(E_{g,n,\beta})=\displaystyle\frac{e_{\C^*}(E_{g,n,\beta}^0)}{e_{\C^*}(E_{0,n,\beta}^1)} \in H^*(\bM_{g,n}(X,\beta))\otimes_\C \C[\lambda,\lambda^{-1}].
%\]
%As a result,
%\[
%e_{\C^*}(E_{g,n,\beta}^1)=\lambda^{r_1}+c_1(E_{g,n,\beta}^1)\lambda^{{r_1}-1}+\cdots{}+c_r(E_{g,n,\beta}^1)\in H^*(\bM_{g,n}(X,\beta))\otimes_{\C}\C[\lambda],
%\]
%where $r_1=\text{rank}(E_{g,n,\beta}^1)$. Since elements in $H^*(\bM_{0,n}(X,\beta))$ are nilpotent, one easily sees that $e_{\C^*}(E_{g,n,\beta})$ is well-defined if we invert $\lambda$. Applying the same reason on $E^0_{g,n,\beta}$, we see $e_{\C^*}(E_{g,n,\beta})$ is invertible if $\lambda$ can be inverted.
Define the twisted Gromov--Witten invariants to be
\begin{equation}\label{eqn:twistedgw}
\langle \psi^{k_1}\alpha_1,\ldots{},\psi^{k_n}\alpha_n \rangle_{g,n,\beta}^{X,E}=\displaystyle\int_{[\bM_{0,n}(X,\beta)]^{vir}} \frac{1}{e_{\C^*}(E_{g,n,\beta})} \cup \prod\limits_{i=1}^n \psi_i^{k_i}ev_{i}^*\alpha_i \in \C[\lambda,\lambda^{-1}] ,
\end{equation}
where $\dfrac{1}{e_{\C^*}(E_{g,n,\beta})}=\dfrac{e_{\C^*}(E_{g,n,\beta}^1)}{e_{\C^*}(E_{g,n,\beta}^0)}$.
\begin{rmk}
For a more detailed discussion about a more general set-up, one can see for example \cite{CG}.
\end{rmk}

\subsection{Theorem \ref{thm:main0} and other consequences}\label{section:thmmain0}

Back to our case, we can apply virtual localizations to $\tilde Y$ and $Y$
respectively. In each case we will index fixed locus by suitable bipartite
graphs. For the assignment of decorated graphs to invariant stable maps, one can
refer to for example \cites{Liu, FL}, among others. Given
$\tilde \Gamma$ a decorated graph for an invariant stable map
$(C,x_1,\dotsc,x_n,f)$ to $\tilde Y$, one can assign a decorated graph to the
image of $(C,x_1,\dotsc,x_n,f)$ under $\tau$ (still invariant because $\tau$ is
equivariant). By a slight abuse of notation, we denote the resulting graph by
$\tau(\tilde \Gamma)$.

One thing to be careful about is that we have chosen to work on some open
substacks. Since $\tau:\bM_{0,n}(\tilde Y,\pi^!\beta) \rightarrow
\bM_{0,n}(Y,\beta)$ is proper, its base-change $\tau:\tilde\scrU\rightarrow
\scrU$ is also proper. Therefore its push-forward on Chow groups are defined.
Localization formula alone may not produce meaningful computational result,
because $\tilde\scrU$ and $\scrU$ are not proper. We need the following
``correspondence of residues".

\begin{prop}\label{prop:res}
Fixing a connected component $\bM_\Gamma$ of the fixed locus in $\scrU$, Theorem
\ref{thm:vircls} implies the following equality in $A_*^{\C^*}(\bM_{\Gamma})$.
\[
\tau_* \left(\sum_{\tau(\tilde \Gamma)=\Gamma} \displaystyle\frac{[\bM_{\tilde\Gamma}]^{vir}}{e_{\C^*}(N^{vir}_{\tilde\Gamma})} \right) =  \displaystyle\frac{[\bM_{\Gamma}]^{vir}}{e_{\C^*}(N^{vir}_{\Gamma})},
\]
where we index fixed components of $\bM_{0,n}(\tilde Y,\beta)$ and $\bM_{0,n}(Y,\beta)$ by $\bM_{\tilde\Gamma}$ and $\bM_{\Gamma}$ respectively (with $\tilde\Gamma$, $\Gamma$ corresponding decorated graphs). $N_{\tilde\Gamma}^{vir}$ and $N_{\Gamma}^{vir}$ are the corresponding virtual normal bundles as well.
\end{prop}

Note that this proposition yields relations between numerical invariants. It is
because in our case, $\bM_\Gamma$ and $\bM_{\tilde \Gamma}$ are all proper. A quick proof of the ``correspondence of residues'' is included in
Appendix \ref{app:res}.

Note that different choices of $\Gamma$ yield different equalities according to
the above corollary. We will show that suitable choices of $\Gamma$ give rise to
neat relations between Gromov--Witten invariants of $\Pp_X(V)$ and (twisted)
Gromov--Witten invariants of $X$. In the rest of this section, we will only show
what choices of $\Gamma$ lead to what relations. Some details of the
computations of $e_{\C^*}(N^{vir}_{\Gamma})$ can be found in Appendix \ref{app:comp}

Theorem \ref{thm:main0} can be deduced by the following choice of graph. Let $\Gamma$ be the graph consisting of a single vertex over $Y_0$ of class $\beta$ with $n$ markings without any edge. Then we have

\begin{equation}\label{eqn:rel1}
\langle \pi^*\sigma_1,\ldots, \pi^*\sigma_n \rangle_{0,n,\pi^!\beta}^{\Pp_X(V), \sO(-1)}=\langle \sigma_1,\ldots, \sigma_n \rangle_{0,n,\beta}^{X, V}.
\end{equation}
This is exactly Theorem \ref{thm:main0}.

%This gives a neat way to recover genus-$0$ Gromov--Witten theory of the base from that of the bundle. One can write it in a more explicit way. It can be checked that $(\sO(-1))_{0,n,\pi^!\beta}$ is a rank $1$ vector bundle over $\bM_{0,n}(\Pp(V),\pi^!\beta)$. Denote $D=c_1((\sO(-1))_{0,n,\pi^!\beta})$. Then Equation \eqref{eqn:rel1} implies
%\[
%\langle \sigma_1,\ldots, \sigma_n \rangle_{0,n,\beta}^{X}=\displaystyle\int_{[\bM_{0,n}(\Pp(V),\pi^!\beta)]^{vir}} (-D)^{r}\cup \prod\limits_{i=1}^n ev_i^*\sigma_i,
%\]
%where $r=rank(V)+(det(V),\beta)-1$. One can take one step further by noticing that $-D=c_1((\sO(1))_{0,n,\pi^!\beta})$, thus eliminating some minus signs in this relation.
One can choose other graphs to obtain other types of relations. We will list a few examples. Let $f\in N_1(\Pp(V))$ be the class of a line in a fiber. Let $\Gamma$ be the graph consisting of a vertex over $Y_0$ of class $\beta$, a vertex over $Y_\infty$ of degree $0$ and an edge of class $f$ connecting them. Put $n$ markings on the vertex over $Y_0$ and $1$ on the one over $Y_\infty$. Then we have
\begin{cor}
\begin{align}\label{eqn:rel2}
\begin{split}
& \langle \dfrac{h^e\pi^*\alpha}{(h-\lambda)(\lambda-h-\psi)}, \pi^*\sigma_1,\ldots, \pi^*\sigma_n \rangle_{0,n+1,\pi^!\beta+f}^{\Pp_X(V), \sO(-1)} \\
=& \langle \pi_*\left(\dfrac{h^e\pi^*\alpha}{(h-\lambda)(\lambda-h-\psi)}\right), \sigma_1,\ldots, \sigma_n \rangle_{0,n+1,\beta}^{X, V},
\end{split}
\end{align}
where $h=c_1(\sO_{\Pp(V)}(1))$.
\end{cor}

One can similarly consider the case when $\Gamma$ consists of a vertex over $Y_0$ and $k$ edges, each of class $f$, coming out of this vertex. For each edge one attach a vertex at the other end over $Y_\infty$ of degree $0$ with $1$ marking. One gets
\begin{cor}
\begin{align}\label{eqn:rel3}
\begin{split}
& \langle \dfrac{h^e\pi^*\alpha_1}{(h-\lambda)(\lambda-h-\psi)}, \ldots, \dfrac{h^e\pi^*\alpha_m}{(h-\lambda)(\lambda-h-\psi)},\pi^*\sigma_1,\ldots, \pi^*\sigma_n \rangle_{0,n+m,\pi^!\beta+kf}^{\Pp_X(V), \sO(-1)} \\
=& \langle \pi_*\left(\dfrac{h^e\pi^*\alpha_1}{(h-\lambda)(\lambda-h-\psi)}\right), \ldots, \pi_*\left(\dfrac{h^e\pi^*\alpha_m}{(h-\lambda)(\lambda-h-\psi)}\right), \sigma_1,\ldots, \sigma_n \rangle_{0,n+m,\beta}^{X, V}.
\end{split}
\end{align}
\end{cor}
One can also consider the $\Gamma$ consisting of a vertex over $Y_0$, an edge of class $kf$, and a vertex of degree $0$ at $Y_\infty$ end with $1$ marking. In this case, we get
\begin{cor}
\begin{align}\label{eqn:rel4}
\begin{split}
&\langle \dfrac{h^e\pi^*\alpha}{\left( \dfrac{(\lambda-h)}{k}-\psi \right)} \dfrac{1}{\prod\limits_{m=1}^{k} \left( \dfrac{m}{k}(h-\lambda) \right) \prod\limits_{m=1}^{k-1} \left(\dfrac{m}{k}(\lambda-h) \right) } ,
\pi^*\sigma_1,\ldots, \pi^*\sigma_n \rangle_{0,n+1,\pi^!\beta+kf}^{\Pp_X(V), \sO(-1)} \\
=&\langle \pi_*\left(\dfrac{h^e\pi^*\alpha }{\left(\dfrac{(\lambda-h)}{k}-\psi \right) \prod\limits_{m=1}^k \left( \dfrac{m}{k}(h-\lambda)\right) \prod\limits_{m=1}^{k-1}\left(c_V(h+\dfrac{m}{k}(\lambda-h))\right)  }\right), 
\sigma_1,\ldots, \sigma_n \rangle_{0,n+1,\beta}^{X, V}.
\end{split}
\end{align}
\end{cor}

\subsection{Completing the proof of Theorem \ref{thm:vircls2}}\label{sec:vircls2}
Theorem \ref{thm:vircls2} is true if and only if the ``correspondence of
residues" in Proposition \ref{prop:res} holds for all $\bM_\Gamma$ no matter
whether it is inside $\scrU$ or not (see Remark \ref{rmk:iff}). Now that Theorem \ref{thm:vircls} is proven, we will show that the correspondence of residues over $\scrU$ is enough to establish the correspondence of all residues.

For each vertex $v$ in a decorated graph, we denote the valence of $v$ by $val(v)$ and the number of markings on $v$ by $n(v)$.

Choose any $\bM_\Gamma$. The decorated graph $\Gamma$ might involve vertices over $Y_\infty$ with nontrivial degrees. Suppose $v_1,\ldots, v_l$ are those vertices with degree $\beta_1,\ldots,\beta_l$ correspondingly. Let $\Gamma(v_i)$ be the graph consisting of a single vertex $v_i$ of degree $\beta_i$ with $val(v_i)+n(v_i)$ markings. In the meantime, we can break the graph $\Gamma$ into pieces $\Gamma_1,\ldots,\Gamma_m$ along $v_1,\ldots,v_l$, where the original vertices of $v_1,\ldots, v_l$ in each $\Gamma_i$ are replaced by unstable vertices each of degree $0$ with $1$ marking. To sum it up, we just break the decorated graph $\Gamma$ into $\Gamma(v_1),\ldots,\Gamma(v_l)$ and $\Gamma_1,\ldots,\Gamma_m$. By gluing the corresponding markings of $\Gamma(v_1),\ldots,\Gamma(v_l), \Gamma_1,\ldots,\Gamma_m$, we can recover $\Gamma$.

By definition we have $\bM_{\Gamma_i}\in \scrU$. In terms of fixed loci of moduli space, we have
\[
\bM_\Gamma=(\prod_{i=1}^m \bM_{\Gamma_i}) \times_\Gamma (\prod_{j=1}^l \bM_{0,val(v_j)+n(v_j)}(Y_\infty, \beta_j)),
\]
where $\times_\Gamma$ is the fiber product gluing the corresponding markings
according to the splitting of the graph $\Gamma$. In $\bM_{0,n}(\tilde
Y,\beta)$, the fixed locus $\bM_{\tilde\Gamma}$ such that
$\tau(\tilde\Gamma)=\Gamma$ can be described by a similar gluing process. We have
\[
\bM_{\tilde \Gamma}=(\prod_{i=1}^m \bM_{\tilde \Gamma_i}) \times_{\tilde\Gamma} (\prod_{j=1}^l \bM_{0,val(v_j)+n(v_j)}(\tilde Y_\infty, \beta_j)),
\]
where $\tilde \Gamma_i$ is the graph such that $\tau(\tilde\Gamma_i)=\Gamma_i$.

Now $e_{\C^*}(N_{\Gamma}^{vir})$ can be
computed using $e_{\C^*}(N_{\Gamma_i}^{vir})$, $\psi$-classes in
$\bM_{0,val(v_i)+n(v_i)}(Y_\infty, \beta_i))$ (smoothing the nodes glued by
``$\times_\Gamma $"), $TY|_{Y_\infty}$, and $(\sO(-1))_{0,val(v_i)+n(v_i),\beta_i}$ in
$K^0(\bM_{0,val(v_i)+n(v_i)}(Y_\infty, \beta_i)))$. If we use $Res(\Gamma)$ to
denote the localization residue of the graph $\Gamma$, the computation can be
schematically written as follows.
\[
Res(\Gamma)=\displaystyle\int_{\bM_\Gamma}
\prod\limits_{i=1}^mRes(\bM_{\Gamma_i})
\prod\limits_{j=1}^l\prod\limits_{e\in
  E(v_j)}\dfrac{1}{(\lambda+h)/d_e-\psi_{(e,v_j)}}
\prod\limits_{k=1}^lRes(\Gamma(v_k)), 
\]
where $E(v_j)$ is the set of edges incident to $v_j$, $d_{e}$ is the degree of
the edge $e$, $\psi_{(e,v_j)}$ is the psi-class in $\bM_{\Gamma(v_j)}$
corresponding to the marking that was supposed to glue with edge $e$ in the
original graph $\Gamma$. Note that if $v_j$ is an
unstable vertex, the corresponding part of the formula needs a slight modification. This is standard and
we do not spell it out to introduce unnecessary notations. 

Similar thing happens to
the localization residue of $\tilde\Gamma$ as well.
\[
Res(\tilde\Gamma)=\displaystyle\int_{\bM_{\tilde\Gamma}}
\prod\limits_{i=1}^mRes(\bM_{\tilde\Gamma_i})
\prod\limits_{j=1}^l\prod\limits_{e\in
  E(v_j)}\dfrac{1}{(\lambda+h)/d_e-\psi_{(e,v_j)}}
\prod\limits_{k=1}^lRes(\Gamma(v_k)). 
\]
Here we use the fact that $Y_{\infty}\cong \tilde Y_{\infty}$, and identify the
residues of vertices over $Y_{\infty}$ with the ones over $\tilde Y_\infty$. By using the correspondence of
residues between $\bM_{\Gamma_i}$ and $\bM_{\tilde\Gamma_i}$ and apply
projection formula to the rest of the factors, one can see that
the localization residues for $\bM_{\Gamma}$ and $\bM_{\tilde\Gamma}$ match.

\section{$\Pp^n$-fibrations}\label{sec:Pnfibr}
Let $\pi: P\rightarrow X$ be a smooth morphism whose fibers are isomorphic to projective spaces. Then $P$ is a \emph{Brauer--Severi} scheme over $X$. We would like to recall a few standard facts about Brauer--Severi schemes. According to \cite{Gro_Br}[I, Th\'eor\`eme 8.2], it is \'etale-locally a product with projective space. It also induces a class in $H^2_{\text{\'et}}(X,\mathbb G_m)$ called \emph{Brauer class}. Let's denote it by $\alpha(P)$. $\alpha(P)=0$ if and only if the fibration $\pi:P\rightarrow X$ comes from projectivization of a vector bundle. We also have the following standard fact.
\begin{lem}
$\tilde\pi:P\times_X P\rightarrow P$ is the projectivization of a vector bundle over $P$.
\end{lem}
This can be seen by, for example, observing $\tilde\pi$ has a section, and then apply \cite{B}[Lemma 2.4]. As a result, we also have the following.
\begin{lem}\label{lem:brauer0}
The induced pull-back $\pi^*:H^2_{\text{\'et}}(X;\mathbb G_m) \rightarrow H^2_{\text{\'et}}(P;\mathbb G_m)$ sends $\alpha(P)$ to $0$.
\end{lem}

Now we want to prove Theorem \ref{thm:main1}. Let's recall the set-up. Let $X$ be a smooth projective variety, $\pi_1:P_1\rightarrow X$, $\pi_2:P_2\rightarrow X$ be projective smooth morphisms whose fibers are isomorphic to $\Pp^n$. Suppose 
\[
\alpha(P_1)=\alpha(P_2)\in H^2_{\text{\'et}}(X;\mathbb G_m).
\]
We also assume the existence of a ring isomorphism 
\[
\scrF:H^*(P_1;\Q) \rightarrow H^*(P_2;\Q),
\]
such that
\begin{enumerate}
\item $\scrF$ sends the subgroup $H^*(P_1;\Z)/tor(H^*(P_1;\Z))$ into $H^*(P_2;\Z)/tor(H^*(P_2;\Z))$. Here $tor(-)$ denotes the torsion subgroup.
\item $\scrF$ restricts to identity on $H^*(X;\Q)$. Here we identify $H^*(X;\Q)$ as subrings under pull-backs of $\pi_1,\pi_2$.
\item $\scrF(c_1(\omega_{\pi_1}))=c_1(\omega_{\pi_2})$, where $\omega_{\pi_i}$ is the corresponding relative canonical sheaf.
\end{enumerate}

\begin{lem}\label{lem:projbdl}
If $\alpha(P_1)=\alpha(P_2)=0$, the existence of such $\scrF$ is equivalent to the existence of vector bundles $V_1, V_2$ over $X$ such that $c(V_1)=c(V_2)$ and $P_1\cong\Pp(V_1), P_2\cong\Pp(V_2)$.
\end{lem}
\begin{proof}
The existence of such $V_1,V_2$ clearly induces such an $\scrF$. It suffices to prove the other direction. Since $\alpha(P_1)=0$, there exists a $V_1$ such that $P_1=\Pp(V_1)$. We have $\omega_{\pi_1}=-\sO_{\Pp(V_1)}(n+1)-\pi_1^*det(V_1)$. The same happens for $P_2$ and we have $P_2=\Pp(V_2)$ and $\omega_{\pi_2}=-\sO_{\Pp(V_2)}(n+1)-\pi_1^*det(V_2)$. $\scrF$ preserving relative canonical sheaf implies that 
\[
\scrF(c_1(\sO_{\Pp(V_1)}(1)))=c_1(\sO_{\Pp(V_2)}(1))-\pi_2^*\left( \dfrac{c_1(det(V_1)-det(V_2))}{n+1} \right).
\]
But since everything else is in $H^2(P_2;\Z)/tor(H^2(P_2;\Z))$, we also have that
\[
\pi_2^*\left( \dfrac{c_1(det(V_1)-det(V_2))}{n+1} \right) \in H^2(P_2;\Z)/tor(H^2(P_2;\Z)).
\]
Since $\pi_2^*$ is an embedding on $H^2(X;\Z)/tor(H^2(X;\Z))$, we have
\[
\left( \dfrac{c_1(det(V_1)-det(V_2))}{n+1} \right) \in H^2(X;\Z)/tor(H^2(X;\Z)).
\]
On the other hand, this class lies in $H^{1,1}(X)$. By Lefschetz theorem, there exists a divisor $L\in Pic(X)$ such that $c_1(L)=\dfrac{c_1(det(V_1)-det(V_2))}{n+1}$. If one changes $V_2$ into $V_2\otimes L$, one can verify that its Chern class will match the one with $V_1$ because $\scrF$ is a ring homomorphism.
\end{proof}

In order to reduce the problem to the projectivization of vector bundles, we do a base-change along $\pi_1:P_1\rightarrow X$ ($\pi_2$ works equally well, but we need to fix one here). Consider the following diagram for both $i=1,2$
\[
\xymatrix{
P_1\times_X P_i \ar[r]^-{pr_{2,i}} \ar[d]^{pr_{1,i}} & P_i \ar[d]^{\pi_i} \\
P_1 \ar[r] & X.
}
\]

Because $\alpha(P_1)=\alpha(P_2)$ and Lemma \ref{lem:brauer0}, this class will be mapped to $0$ via pull-back along either $\pi_1$ or $\pi_2$. As a result, the four maps $pr_{1,i}$ and $pr_{2,i}$ for $i=1,2$ all come from the projectivization of vector bundles.

Because of the degeneration of Leray spectral sequence at $E_2$, one can check that $H^*(P_1\times_X P_i;\Q)=H^*(P_i;\Q)\otimes_{H^*(X;\Q)} H^*(P_1;\Q)$. The ring isomorphism $\scrF$ passes to $P_1\times_X P_i$ over $P_1$ and all conditions are preserved (by slight abuse of notation, we still use $\scrF$ for it). By Lemma \ref{lem:projbdl} and the main theorem in \cite{F}, the Gromov--Witten theory of $P_1\times_X P_1$ and $P_1\times_X P_2$ are identified via $\scrF$. To be precise,
\[
\langle \psi^{k_1}\sigma_1,\dotsc,\psi^{k_n}\sigma_n \rangle_{0,n,\beta}^{P_1\times_X P_1}=\langle \psi^{k_1}\scrF\sigma_1,\dotsc,\psi^{k_n}\scrF\sigma_n \rangle_{0,n,\Psi(\beta)}^{P_1\times_X P_2}
\]

Finally, we want to relate back via $pr_{2,i}$ to prove such identification for $P_1$ and $P_2$. Suppose $P_1\times_X P_i=\Pp_{P_i}(V_i)$ for $i=1,2$. Because we can tensor each of them with $det(T_{\pi_i})\otimes \pi_i^*L$ where $L$ is a sufficiently ample line bundle, we can assume both $V_1$ and $V_2$ are globally generated. Now we are at a place to apply our Corollary \ref{cor:main}. Since twisted Gromov--Witten invariant can be computed by descendant Gromov--Witten invariants, we also have
\[
\langle \psi^{k_1}\sigma_1,\dotsc,\psi^{k_n}\sigma_n \rangle_{0,n,\beta}^{\Pp_{P_1}(V_1),\sO(-1)}=\langle \psi^{k_1}\scrF\sigma_1,\dotsc,\psi^{k_n}\scrF\sigma_n \rangle_{0,n,\Psi(\beta)}^{\Pp_{P_2}(V_2),\sO(-1)}.
\]
Immediately applying Corollary \ref{cor:main}, we have
\[
\langle \sigma_1,\dotsc,\sigma_n \rangle_{0,n,\beta}^{P_1}=\langle \scrF\sigma_1,\dotsc,\scrF\sigma_n \rangle_{0,n, \Psi(\beta)}^{P_2}.
\]
This is almost the form of Theorem \ref{thm:main1} except we don't have descendants here. However, since we only work on genus-$0$ Gromov--Witten theory, descendants invariants can be computed via the topological recursion relation from the absolute invariants. Since the cohomology rings of $H^*(P_1;\Q)$ and $H^*(P_2;\Q)$ are identified, each step in the computation can be matched. This concludes our proof for Theorem \ref{thm:main1}.

\appendix
\section{Correspondence of residues}\label{app:res}
Let's put everything into a general setting. Let $f:X\rightarrow Y$ be a
proper map between Deligne--Mumford stacks. Suppose $T=\C^*$ acts on
both $X$ and $Y$ and $f$ is equivariant. Let the connected components of the fixed loci of $X$ and $Y$ be
$X_1,\dotsc,X_n$ and $Y_1,\dotsc,Y_m$, respectively. Denote
$i_{X_j}:X_j\hookrightarrow X$, $i_{Y_i}:Y_i\hookrightarrow Y$ as the
corresponding embeddings. Suppose that there are
equivariant perfect obstruction theories of $X$ and $Y$. They induce natural
perfect obstruction theories on components $X_j$,$Y_i$ (\cite{GP}). We want to prove the following theorem. 
\begin{prop}
Assume $X$ and $Y$ admit equivariant embeddings into smooth Delign-Mumford
stacks with $\C^*$ actions. Suppose $f_*$ preserves equivariant virtual classes, i.e., $f_*[X]^{vir,eq}=[Y]^{vir,eq}$.
Fixing a connected component $Y_i$ of the fixed locus in $Y$. We have
the following equality in $A_*^{\C^*}(Y_i)$.
\[
f_* \left(\sum\limits_{f(X_j)\subset Y_i} \displaystyle\frac{[X_j]^{vir}}{e_{\C^*}(N^{vir}_{X_j/X})} \right) =  \displaystyle\frac{[Y_i]^{vir}}{e_{\C^*}(N^{vir}_{Y_i/Y})},
\]
where $N^{vir}_{X_j/X}$ and $N^{vir}_{Y_i/Y}$ are the corresponding virtual
normal bundles, and we slightly abuse the notation by writing $f|_{f^{-1}(Y_i)}$
simply as $f$.
\end{prop}
\begin{proof}
  This is basically a direct consequence of the decomposition of equivariant
  Chow groups in \cite[Theorem 5.3.5]{Kr} and the explicit localization formula
  in \cite{GP}. More precisely, by \cite[Theorem 5.3.5]{Kr}, the equivariant Chow groups of $X$ and $Y$ admit
  the following decompositions.
  \[
    A^{\C^*}_*(X) \otimes_{\Q} \Q[\lambda,\lambda^{-1}] = \bigoplus\limits_{j=1}^n A^{\C^*}_*(X_j) \otimes_{\Q[\lambda]} \Q[\lambda,\lambda^{-1}],
  \]
  \[
    A^{\C^*}_*(Y) \otimes_{\Q} \Q[\lambda,\lambda^{-1}] = \bigoplus\limits_{i=1}^m A^{\C^*}_*(Y_i) \otimes_{\Q[\lambda]} \Q[\lambda,\lambda^{-1}].
  \]
  Since $f$ maps fixed locus into fixed locus, $f_*$ respects these
  decompostions. Now \cite{GP} tells us that under the above decompositions, the
  $X_j$ component of $[X]^{vir,eq}$ can be explicitly written as
  $\dfrac{[X_j]^{vir}}{e_{\C^*}(N^{vir}_{X_j/X})}$. Since $f_*$ respects the
  decomposition, the $Y_i$ component of
  $f_*[X]^{vir,eq}$ is $f_* \left(\sum\limits_{f(X_j)\subset Y_i} \displaystyle\frac{[X_j]^{vir}}{e_{\C^*}(N^{vir}_{X_j/X})} \right)$. On
  the other hand, the $Y_i$ component of $[Y]^{vir,eq}$ is
  $\dfrac{[Y_i]^{vir}}{e_{\C^*}(N^{vir}_{Y_i/Y})}$. Since
  $f_*[X]^{vir,eq}=[Y]^{vir,eq}$, the result follows.
\end{proof}
\begin{rmk}\label{rmk:iff}
  Conversely, if all localization residues match, one can conclude that
  $f_*[X]^{vir,eq}=[Y]^{vir,eq}$ by similar argument. Again, the
  decomposition \cite[Theorem 5.3.5]{Kr} is the key. 
\end{rmk}

\section{Computing the virtual normal bundles}\label{app:comp}
Recall that $Y=\Pp_X(V\oplus \sO)$, $\tilde Y=\Pp_{\Pp(V)}(\sO(-1)\oplus \sO)$,
and the map $\pi:\tilde Y\rightarrow Y$ contracts the zero section $\tilde
Y_0=\Pp_{\Pp(V)}(\sO)\subset \tilde Y$ to the zero section $Y_0=\Pp_X(\sO)\subset Y$.
We would like to give some extra explanations to the derivation of equation \eqref{eqn:rel1}-\eqref{eqn:rel4}.
\subsection{A single vertex over $Y_0$}
Now $\Gamma$ consists of a single vertex over $Y_0$. First we want to compute
$e_{\C^*}(N^{vir}_{\Gamma})$. Recall there exists a two term complex
$[E^{-1}\rightarrow E^0$] consisting
of vector bundles resolving the relative perfect obstruction theory of
$\bM_{0,n}(Y,\beta)$ over the moduli of prestable curves. Note that $E^\bullet$
is quasi-isomorphic to $(R^\bullet(ft_{n+1})_*ev^*_{n+1}TY)^\vee$ (Notations see
Section \ref{sec:twistedgw}). Since $Y_0$ is the fixed locus and the curve
completely maps to $Y_0$, one can see that the moving part of $(E^\bullet)^\vee$
is a complex resolving $R^\bullet(ft_{n+1})_*ev^*_{n+1}N_{Y_0/Y}$. Comparing
with the definition in Section \ref{sec:twistedgw}, the localization residue is
exactly the Gromov--Witten invariants twisted by $V\cong N_{Y_0/Y}$. To be
precise, if we fix the insertions $\sigma_1,\dotsc,\sigma_n\in H^*(X)$, the
pushforward of the localization residue of $\Gamma$ to a point results in the
following invariant (Also see the right hand side of equation \eqref{eqn:rel1}). 
\[
\langle \sigma_1,\ldots, \sigma_n \rangle_{0,n,\beta}^{X, V}.
\]

On the other hand, one sees that there is only one graph $\tilde\Gamma$ that lifts
$\Gamma$. $\tilde\Gamma$ consists of only one vertex over $\tilde Y_0$. By a
similar argument, the corresponding localization residue gives rise to the Gromov--Witten
invariant of $\Pp_X(V)\cong \tilde Y_0$ twisted by $\sO(-1)$. To be precise, it
is
\[
\langle \pi^*\sigma_1,\ldots, \pi^*\sigma_n \rangle_{0,n,\pi^!\beta}^{\Pp_X(V), \sO(-1)}.
\]

\subsection{A vertex with edges}
Now suppose the graph $\Gamma$ consists of a vertex $v_1$ over $Y_0$ of class
$\beta$ with $n$ markings,
a vertex $v_2$ over $Y_{\infty}$ of class $0$ (i.e. degree $0$) with $1$ marking, and an edge of class
$f$ (the curve class of line inside a fiber) connecting these two vertices.
Label the marking on $v_2$ as the $1$-st marking, and the $n$ markings on
$v_1$ as $2,\dotsc,n+1$-th markings. On the $1$-st marking, we choose
$h^e\pi^*\alpha$ as our insertion. On the $2,\dotsc,n+1$-th markings,
$\pi^*\sigma_1,\dotsc,\pi^*\sigma_n$ are the corresponding insertions.

The lifting of the graph $\Gamma$ is a graph $\tilde\Gamma$ with a vertex
$\tilde v_1$ over $\tilde Y_0$ of class $\pi^!\beta+f$, a vertex $\tilde v_2$
over $\tilde Y_\infty$ of class $0$ and an edge between them.

The explicit localization formula for an arbitrary $\C^*$-action is already
studied in different papers, for example, \cite{MM} and \cite{FL}, among others.
We have
\[
  Res(\tilde\Gamma)=\displaystyle\int_{\bM_{\tilde\Gamma}}
  \dfrac{1}{e_{\C^*}(\sO(-1)_{0,n+1,\pi^!\beta+f})}\dfrac{ev^*_1(h^e\pi^*\alpha)\prod\limits_{i=1}^n ev^*_{i+1}\pi^*\sigma_n}{ev^*_1(h-\lambda)(ev^*_1(\lambda-h)-\psi_1)}.
\]
Heuristically, $ev^*_1(h-\lambda)$ corresponds to the deformation of the $1$-st
marking, and $ev^*_1(\lambda-h)-\psi_i$ corresponds to the formal deformation of
smoothing the node between $\tilde v_1$ and the edge. We note that
$\bM_{\tilde\Gamma}\cong\bM_{0,n+1}(\Pp(V),\pi^!\beta+f)$. Fiber product does
not affect the space because the target of
$ev_1$ and the deformation space of the edge are both $\Pp(V)$. 

Due to the similarity of the graphs, we have a similar expression for the residue of $\Gamma$.
\[
  Res(\Gamma)=\displaystyle\int_{\bM_{\Gamma}}
  \dfrac{1}{e_{\C^*}(V_{0,n+1,\beta+f})}\dfrac{ev^*_1(h^e\pi^*\alpha)\prod\limits_{i=1}^n ev^*_{i+1}\pi^*\sigma_n}{ev^*_1(h-\lambda)(ev^*_1(\lambda-h)-\psi_1)}.
\
\]
But note that $\bM_{\Gamma}\cong \bM_{0,n+1}(X,\beta)
\times_{(ev_1,\pi)} \Pp_X(V)$, where $\times_{(ev_1,\pi)}$ means the fiber
product of $ev_1$ and $\pi$. Here we abuse notation by using the same $ev_1$
for the evaluation map of $\bM_{0,n+1}(X,\beta)$. We warn the reader that this
evaluation map is only used to glue with the edge. This map is completely
different from the $ev_1$ of $\bM_{\Gamma}$ whose target is $\Pp(V)$ instead
of $X$. 

Although the above integrands are similar, in order to get a numerical result,
we need to push them forward to a point. It's very straightforward for
$\bM_{\tilde\Gamma}$. But for $\bM_{\Gamma}$, there is an extra projection along
the fibers of $\Pp(V)$. This is why we get
\[
\langle \pi_*\left(\dfrac{h^e\pi^*\alpha}{(h-\lambda)(\lambda-h-\psi)}\right), \sigma_1,\ldots, \sigma_n \rangle_{0,n+1,\beta}^{X, V}
\]
on the right hand side of equation \eqref{eqn:rel2}.

For equations \eqref{eqn:rel3} and \eqref{eqn:rel4}, the reasonings are similar.
We note that in order to derive equation \eqref{eqn:rel4}, the $k$-fold
covering of the edge give rise to extra factors in the virtual normal bundles. It
corresponds to a factor
\[
  ev^*_1\dfrac{1}{\prod\limits_{m=1}^k(\dfrac{m}{k}(h-\lambda))\prod\limits_{m=1}^{k-1}(\dfrac{m}{k}(\lambda-h))}
\]
in $\bM_{\tilde\Gamma}$, and to a factor
\[
  ev^*_1\dfrac{1}{\prod\limits_{m=1}^k \left( \dfrac{m}{k}(h-\lambda)\right) \prod\limits_{m=1}^{k-1}\left(c_V(h+\dfrac{m}{k}(\lambda-h))\right)} 
\]
in $\bM_\Gamma$.
The rest of the argument is similar and straightforward.

\newpage
\bibliographystyle{cdraifplain}
\bibliography{bib}

\end{document}